\documentclass[12pt,final,reqno]{amsart}
\usepackage{ntung}

\usepackage{ntung_commands}

\usepackage{richardcommands}

\begin{document}

\title{On efficient graph covers and steered random walks}
\author{Nathan Tung}
\author{Richard Ueltzen}
\address{Stanford University, CA 94305, USA}
\email{$\{$ntung,rueltzen$\}$@stanford.edu}
\begin{abstract}
    We prove that the vertices of any $n$-vertex graph can be partitioned into pieces of radius $r = O(\log n)$ such that the sum of the sizes of their closed neighborhoods is at most $4n$. This answers a recent question of Bukh and Dubroff and directly yields an improvement to their upper bound on the optimal cover time of the $\epsilon$-steered random walk. We also demonstrate that our bound on $r$ is best possible up to a constant factor for graphs with strong vertex expansion.
\end{abstract}
\maketitle

\section{Introduction}

Very recently, Bukh and Dubroff~\cite{bukh-dubroff2026} introduced $(r, k)$-covers for undirected graphs to construct faster random walks via infrequent steering.
For $r \geq 0$ and $k \geq 1$, an $(r, k)$-cover for $G$ is a collection of subsets $\{U_i \subseteq V(G)\}$ with $\bigcup_i U_i = V(G)$ such that each $G[U_i]$ has radius at most $r$, and with $\sum_i |N^{\leq 1}(U_i)| \leq k n$.
Bukh and Dubroff~\cite{bukh-dubroff2026} proved that every $n$-vertex graph admits a $(2^{a-1}-1, n^{1/a+o(1)})$-cover for any positive integer $a$ using an expansion argument and imitating the proof of Vitali's covering lemma.
In particular this shows that any $G$ admits a $(2^{O(\sqrt{\log n})}, 2^{O(\sqrt{\log n})})$-cover. They also provide algebraic constructions to show that their bound for the existence of $(r, k)$-covers is tight for small values of the radius parameter $r$.
Namely, if $k(n, r)$ is the minimal $k$ such that every $n$-vertex graph $G$ admits an $(r, k)$-cover, then they show that
\begin{align*}
    k(n, 1) = n^{1 / 2 + o(1)}, \quad k(n, 2) = n^{1 / 2 + o(1)}, \quad k(n, 3) = n^{1 / 3 + o(1)}.
\end{align*}
Furthermore, they asked whether better covers can be found in the general case.
Using a clustering algorithm from \cite{milleralgo}, we improve their bound for $r \geq 8$.
\begin{restatable}{theorem}{covering}\label{main-theorem}
    For any graph $G$ with $n \geq 1$ vertices and for any $r \geq 1$, there exists a partition $\{U_i\}$ of the vertex set $V(G)$ such that each $G[U_i]$ is connected with radius at most $r$ and $\sum_i |N^{\leq 1}(U_i)| \leq 2 (3n)^{2 / r} n$.
    In particular, $G$ admits a $(r, 2 (3n)^{2 / r})$-cover.
\end{restatable}
Bukh and Dubroff asked if there always exists a $(\poly\log n, \poly\log n)$-cover for any $n$-vertex graph.
Taking $r = 2 \log(3n)$ in \Cref{main-theorem} yields that any $n$-vertex graph has a $(O(\log n), O(1))$-cover, thereby answering their question affirmatively.
\begin{corollary}\label{log-log-cover}
    Any  graph $G$ on $n \geq 1$ vertices admits a $(2 \log(3n), 4)$-cover.
\end{corollary}

\subsection{Application to steered random walks} Starting at some vertex $s \in V(G)$, the $\eps$-steered random walk is identical to the simple random walk except that at each step there is a small probability $\eps > 0$ that a controller gets to choose the next step of the walk (among neighbors of the current vertex). This choice takes the form of a probability distribution $\phi$ over neighbors, which is allowed to consider as input the entire history of the walk. If $\tau_{\phi,s}(\eps)$ is the cover time of this walk, then
$$
T(G,\eps) \coloneqq \inf_\phi \max_{s \in V(G)} \EE \tau_{\phi,s}(\eps)
$$
is the optimal cover time guaranteed by some strategy $\phi$.
We concern ourselves with the quantity
$$
f_{\Delta}(n,\eps) \coloneqq \max_{G} T(G,\eps),
$$
where the maximum is over connected $n$-vertex graphs with maximum degree at most $\Delta$. We borrow this notation from \cite{bukh-dubroff2026}, which may be consulted for formal definitions.

Bukh and Dubroff show in \cite{bukh-dubroff2026} that if every $n$-vertex graph admits an $(r, k)$-cover, then for any natural number $\Delta$ and any $0 < \epsilon \leq 1$,
\begin{equation}
    f_\Delta(n, \epsilon) < 32 \epsilon^{-1} \Delta (r+1) k n \log^2 n.\label{general-bound-for-f-delta-n-epsilon}
\end{equation}
Plugging in their $(2^{O(\sqrt{\log n})}, 2^{O(\sqrt{\log n})})$-covers, their unconditional bound is
\begin{equation*}
    f_\Delta(n, \epsilon) < \epsilon^{-1} \Delta n2^{O(\sqrt{\log n})}.
\end{equation*}
By black-boxing \eqref{general-bound-for-f-delta-n-epsilon} and plugging in \Cref{log-log-cover}, we can remove the quasi-polynomial factor:
\begin{corollary}
    $f_\Delta(n, \epsilon) = O(\epsilon^{-1} \Delta n \log^3 n)$.
\end{corollary}

\subsection{Lower bound from vertex expansion} We also show that strong vertex expanders of maximum degree $\Delta$ with vertex-expansion $\Omega(\Delta)$ admit no local efficient covers.
\begin{restatable}{theorem}{uncoverable}
    \label{other-direction-theorem}
    Let $n, r$ be positive integers with $n$ even.
    Then there exists an $n$-vertex graph that admits no $(r, n^{1 / (r+1)}/15)$-cover.
\end{restatable}
This shows that \Cref{main-theorem} is tight up to a constant factor in front of $r$. Stated differently, \Cref{main-theorem,other-direction-theorem} yield the bounds
\begin{equation*}
    n^{(1 + o(1)) / r} \leq k(n, r) \leq n^{(2 + o(1)) / r}
\end{equation*}
for $\omega(1) \leq r \leq o(\log n)$.

\subsection{Notation and definitions}

All logarithms in this paper are base $2$.
For a positive integer $n$, we denote $\{1, \dots, n\}$ by $[n]$.
By a graph, we mean a finite simple undirected graph.
For a graph $G$, we denote its set of vertices and edges by $V(G), E(G)$, respectively. For $U,V \subseteq V(G)$, $\dist(U,V)$ will denote the shortest path distance from any $u \in U$ to any $v \in V$. By convention $\dist(U,V) = \infty$ if no such path exists.
For $U \subseteq V(G)$, the subgraph induced by $U$ is denoted $G[U]$.
For any $U \subseteq V(G)$ and any $r \geq 0$, define
\begin{equation*}
    N^{\leq r}(U) \coloneqq \{v \in V(G) : \dist(U, \{v\}) \leq r\}.
\end{equation*}
For $u \in V(G)$, we write $N^{\leq r}(u) \coloneqq N^{\leq r}(\{u\})$.
\begin{definition}
    The \define{radius} of a non-empty connected graph $G$ is the minimum $r \geq 0$ such that there exists a $v \in V(G)$ with $N^{\leq r}(v) = V(G)$.
    Further, we set the radius of the empty graph to be $0$.
\end{definition}

\begin{definition}\label{r-k-cover}
    For a graph $G$ and real numbers $r, k$, we say that $U_1, \dots, U_m \subseteq V(G)$ form an \define{$(r, k)$-cover} for $G$ if
    \begin{enumerate}
        \item $U_1, \dots, U_m$ \define{covers} $V(G)$: $V(G) = \bigcup_{i=1}^m U_i$,\label{r-k-cover-is-a-cover}
        \item $U_1, \dots, U_m$ is \define{$r$-local}: $G[U_i]$ is connected with radius at most $r$ for all $1 \leq i \leq m$,\label{r-k-cover-has-covering-sets-of-radius-r}
        \item $U_1, \dots, U_m$ is \define{$k$-efficient}: $\sum_{i=1}^m |N^{\leq 1}(U_i)| \leq k |V(G)|$.\label{r-k-cover-is-small}
    \end{enumerate}
\end{definition}

\section{Construction and analysis of nearly optimal $(r, k)$-covers}\label{main-thm-proof-section}

We use a clustering algorithm of Miller, Peng, and Xu \cite{milleralgo}. Miller, Peng, Vladu, and Xu \cite{mpvx} subsequently proved a local intersection estimate for this clustering that controls the number of clusters meeting a metric ball. The proof below is an adaptation of these arguments.

\covering*
\begin{proof}
    The theorem is trivial for $n = 1$, so we assume $n \geq 2$.
    By applying the theorem to each connected component, it suffices to consider the case that $G$ is connected.
    Let us first describe the algorithm for constructing the cover.
    We sample independent identically distributed exponential random variables $H_x$ ($x \in V(G)$) parameterized by $\theta \coloneqq (3n)^{-1 / r}$, meaning
    \begin{equation*}
        \IP(H_x \geq t) = \theta^t
    \end{equation*}
    for any $t \geq 0$.
    Set $S_x(v) \coloneqq H_x - \dist(v, x)$ for $x, v \in V(G)$, and define
    \begin{equation*}
        U_x \coloneqq \left\{v \in V(G) \mmid S_x(v) = \max_{y \in V(G)} S_y(v)\right\}.
    \end{equation*}
    Then $\{U_x \mid x \in V(G)\}$ is indeed a cover for $G$ since, for each $v \in V(G)$, the maximum $\max_{y \in V(G)} S_y(v)$ is attained by some value of $y \in V(G)$. It is in fact a partition almost surely as these maxima are almost surely attained by a unique maximizer.
    We also have that the event
    \begin{equation*}
        \E \coloneqq \{H_x \leq r \quad \forall x \in V(G)\}
    \end{equation*}
    holds with probability
    \begin{equation*}
        \IP(\E) \geq 1 - \sum_{x \in V(G)} \IP(H_x > r) = 1 - n \theta^r = 1 - n \frac{1}{3n} = \frac{2}{3}.
    \end{equation*}
    We claim that if $\E$ holds, then the cover $\{U_x \mid x \in V(G)\}$ is $r$-local. Suppose $\E$ holds and let $x \in V(G)$.
    If $U_x$ is empty then it has radius $0 \leq r$.
    Now suppose $U_x$ is non-empty and let $v \in U_x$ be arbitrary.
    Then
    \begin{equation*}
        H_x - \dist(v, x) = S_x(v) \geq S_v(v) = H_v \geq 0,
    \end{equation*}
    which implies $\dist(v, x) \leq H_x \leq r$ by $\E$.
    Therefore, there exists a shortest $v$-$x$-path $P$ in $G$ of length at most $r$.
    We claim that $V(P) \subseteq U_x$.
    Indeed, for $w \in V(P)$, we have $\dist(v, x) = \dist(v, w) + \dist(w, x)$ since $P$ is a shortest path in $G$. This implies that
    \begin{equation*}
        S_x(w) = S_x(v) + \dist(v, w) = \max_{y \in V(G)} S_y(v) + \dist(v, w) \geq \max_{y \in V(G)} S_y(w),
    \end{equation*}
    so $w \in U_x$.
    Thus, $P$ is a $v$-$x$-path in $G[U_x]$ of length at most $r$, so $x \in U_x$ and $\dist_{G[U_x]}(v, x) \leq r$.
    We have therefore shown that if $\E$ holds, then the cover $\{U_x \mid x \in V(G)\}$ is $r$-local.

    In the remainder of the proof we bound the expected efficiency of the cover $\{U_x \mid x \in V(G)\}$ to show that it is efficient for some outcome satisfying $\E$.
    To that end, we count the number of times a vertex $v \in V(G)$ appears in $N^{\leq 1}(U_x)$ for some $x \in V(G)$. We rewrite
    \begin{equation}\label{double-counting-efficiency-sum}
        \sum_{x \in V(G)} |N^{\leq 1}(U_x)| = \sum_{v \in V(G)} \sum_{x \in V(G)} 1_{N^{\leq 1}(v) \cap U_x \neq \emptyset}.
    \end{equation}
    Suppose now that the event $\{N^{\leq 1}(v) \cap U_x \neq \emptyset\}$ holds. Then there exists a $w \in N^{\leq 1}(v) \cap U_x$. 
    Then by $\dist(v, w) \leq 1$, we get $|S_y(v) - S_y(w)| \leq 1$ for any $y \in V(G)$.
    Therefore,
    \begin{equation*}
        S_x(v) \geq S_x(w) - 1 = \max_{y \in V(G)} S_y(w) - 1 \geq \max_{y \in V(G)} S_y(v) - 2 \geq \max_{y \in V(G) \backslash \{x\}} S_y(v) - 2.
    \end{equation*}
    We conclude that
    \begin{equation*}
        \{N^{\leq 1}(v) \cap U_x \neq \emptyset\} \subseteq \set{S_x(v) \geq \max_{y \in V(G) \backslash \{x\}} S_y(v) - 2}.
    \end{equation*}
    Therefore, for every $v \in V(G)$,
    \begin{align*}
        \IE \left(\sum_{x \in V(G)} 1_{N^{\leq 1}(v) \cap U_x \neq \emptyset}\right)
        &\leq \sum_{x \in V(G)} \IP\left(S_x(v) \geq \max_{y \in V(G) \backslash \{x\}} S_y(v) - 2\right)\\
        &\leq \frac{1}{\theta^2} \sum_{x \in V(G)} \IP\left(S_x(v) > \max_{y \in V(G) \backslash \{x\}} S_y(v)\right),
    \end{align*}
    where we used the independence of $H_x$ and $(H_y)_{y \in V(G) \backslash \{x\}}$, and the memorylessness property of the exponential distribution: For any $t \in \IR$ and $x \in V(G)$,
    \begin{equation*}
        \IP(H_x > t) = \theta^{\max\{0, t\}} \geq \theta^2 \theta^{\max\{0, t-2\}} = \theta^2 \IP(H_x \geq t - 2).
    \end{equation*}
    For each fixed $v \in V(G)$ and any outcome, there exists at most one $x \in V(G)$ with $S_x(v) > \max_{y \in V(G) \backslash \{x\}} S_y(v)$.
    Therefore, with \eqref{double-counting-efficiency-sum}, we obtain
    \begin{equation*}
        \IE\left(\sum_{x \in V(G)} |N^{\leq 1}(U_x)|\right) \leq \frac{1}{\theta^2} \sum_{v \in V(G)} \sum_{x \in V(G)} \IP\left(S_x(v) > \max_{y \in V(G) \backslash \{x\}} S_y(v)\right) \leq \frac{n}{\theta^2}.
    \end{equation*}
    Then by Markov's inequality,
    \begin{equation*}
        \IP\left(\sum_{x \in V(G)} |N^{\leq 1}(U_x)| \leq \frac{2n}{\theta^2}\right) \geq \frac{1}{2}.
    \end{equation*}
    Since $\IP(\E) \geq \frac{2}{3} > \frac{1}{2}$, there exists an outcome for which $\E$ holds and $\sum_{x \in V(G)} |N^{\leq 1}(U_x)| \leq \frac{2n}{\theta^2}$ and for which every maximum $\max_x S_x(v)$ is uniquely attained.
    For that outcome, $\{U_x \mid x \in V(G), U_x \neq \emptyset\}$ is a partition of $V(G)$ and an $(r, 2/\theta^2) = (r, 2 (3n)^{2 / r})$-cover for $G$, as required.
\end{proof}

\subsection{Tightness of the analysis}\label{non-optimality-of-algo}

\Cref{main-theorem} is an existential result, but its proof uses a concrete randomized algorithm. It is natural to ask whether this algorithm is instance-adaptive: when a graph admits much better local efficient covers, does the randomized construction typically find one? The answer is no. The examples below show that the two estimates in the proof of \Cref{main-theorem} are essentially tight for the randomized construction itself: paths witness tightness of the radius estimate, while cliques witness tightness of the efficiency estimate up to a factor of $2$ in the exponent.

\begin{proposition}
    Let $G$ be a path with vertex set $[n]$, let $2 \le r \leq n^{o(1)}$, and let $C > 0$ be a fixed constant. Then with high probability the above construction with exponential parameter $\theta = (Cn)^{-1/r}$ yields a cluster $U_x$ $(x \in [n])$ so that $G[U_x]$ has radius at least $\lfloor (1-o(1))r \rfloor$.
\end{proposition}
\begin{proof}
    Set $\eps \coloneqq \sqrt{(\log r)/\log n}$ and note that $\epsilon = o(1)$. Denote $M \coloneqq \max_{x \in [n]} H_x$, and choose $x_0 \in [n]$ with $H_{x_0} = M$.
    Almost surely, there is a unique choice for $x_0$.
    Since $\theta^r = C^{-1} n^{-1}$ and $\epsilon = \Omega((\log n)^{-1 / 2})$,
    \begin{equation*}
        \IP(M < (1-\epsilon)r)
        =
        (1-\theta^{(1-\epsilon)r})^n
        \leq
        \exp(-n\theta^{(1-\epsilon)r})
        =
        \exp(-\Omega(n^\epsilon))
        =
        o(1).
    \end{equation*}
    Note also that since $x_0$ is uniformly distributed on $[n]$,
    \begin{equation*}
        \IP(\dist(x_0,\{1,n\}) < r) \leq \frac{2r}{n} = o(1).
    \end{equation*}
    With high probability we may thus condition on $x_0$ being at least $r$ away from the endpoints and $M \ge (1-\eps)r$. Conditioned on $x_0$ and $M$, the variables $(H_y)_{y \neq x_0}$ are distributed as independent copies of $H_y$ conditioned to be less than $M$.
    For sufficiently large $n$,
    \begin{equation*}
        \IP\bigp{H_y \geq \epsilon r \mid H_y < M} = \frac{\IP\bigp{\eps r \le H_y < M}}{\IP\bigp{H_y < M}} = \frac{\max\{0, \theta^{\epsilon r} - \theta^M\}}{1-\theta^M} \leq \theta^{\epsilon r}.
    \end{equation*}
    By a union bound,
    \begin{equation*}
        \IP(\exists y \neq x_0 : \dist(y,x_0) \leq 2r,\ H_y \geq \epsilon r)
        \leq
        O(r)\theta^{\epsilon r}
        =
        O(rn^{-\epsilon})
        =
        o(1),
    \end{equation*}
    using the definition of $\eps$.

   Condition on the complement of this event, and fix an arbitrary $v \in N^{\le \floor{(1-2\eps)r}}(x_0)$. To lower bound the radius of $U_{x_0}$ by $\floor{(1 - 2\eps)r}$ it suffices to show $v \in U_{x_0}$.
   Consider any $y \in [n]$ with $y \neq x_0$. If $\dist(y,x_0) \le 2r$, then
   $$
   S_y(v) < \eps r - \dist(y,v) \le \eps r \le (1-\eps)r - \floor{(1-2\eps)r} \le S_{x_0}(v).
   $$
   If $\dist(y,x_0) > 2r$, then
   $$
   S_y(v) < M - \dist(y,v) < M - \dist(x_0,v) = S_{x_0}(v).
   $$
   In either case $S_y(v) < S_{x_0}(v)$. Therefore, $S_{x_0}(v) = \max_y S_y (v)$, so $v \in U_{x_0}$, as desired.
\end{proof}

The following standard fact from probability can be derived from the order statistics of exponential random variables \cite{exponential-order-statistics}.

\begin{fact}\label{fact:orderstats}
    Let $0 < \theta < 1$ and let $X_1,\dots, X_m$ be i.i.d. exponential random variables with $\prob{X_i \geq t} = \theta^t$ for $t \geq 0$.
    If $W_m$ denotes the number of $i \in [m]$ such that $X_i \ge \max_j X_j - 1$, then $W_m \in [m]$ is a truncated geometric random variable with
    $$
    \prob{W_m \ge k} = (1-\theta)^{k-1}, \quad 1 \le k \le m.
    $$
\end{fact}

\begin{proposition}
    Let $r \geq 2$ and let $G$ be the disjoint union of $m \geq 2$ cliques each with $m$ vertices, so $G$ has $n = m^2$ vertices. Then with high probability the above construction with exponential parameter $0 < \theta \le n^{-1/r}$ yields a cover satisfying $\sum_{x \in V(G)} |N^{\leq 1}(U_x)| \geq \Omega(n^{1 + 1 / r})$.
\end{proposition}
\begin{proof}
    Let the cliques be $C_1,\dots,C_m$ and for every $i \in [m]$ let $W_i$ denote the number of $x \in V(C_i)$ such that $U_x \neq \emptyset$ after running the partitioning algorithm.
    For $x \in V(C_i)$, we have that $U_x \neq \emptyset$ if and only if $H_x \geq \max_{y \in V(C_i)} H_y - 1$, and then $N^{\leq 1}(U_x) = V(C_i)$.
    Then by \Cref{fact:orderstats} and by $\theta \leq m^{-2 / r}$,
    $$
    \prob{W_i \ge \floor{m^{2/r}}} = (1-\theta)^{\floor{m^{2/r}} - 1} \ge \left(1-m^{-2/r}\right)^{m^{2/r}} \geq \frac{1}{4} = \Omega(1)
    $$
    if $m^{2 / r} \geq 2$. If $m^{2 / r} < 2$, then $\prob{W_i \ge \floor{m^{2/r}}} = \prob{W_i \geq 1} = 1 = \Omega(1)$. Therefore, $\prob{W_i \geq \floor{m^{2 / r}}} = \Omega(1)$ holds in any case.
    Thus, the number of $i$ such that $W_i \ge \floor{m^{2/r}}$ is a binomial random variable with expectation $\Omega(m)$. As a result, the number of such $i$ is $\Omega(m)$ with high probability.
    On that event,
    \begin{equation*}
        \sum_{x \in V(G)} |N^{\leq 1}(U_x)| \geq \sum_{\substack{i \in [m] \\ W_i \geq \floor{m^{2 / r}}}} m W_i \geq \Omega(m^2 \floor{m^{2 / r}}) = \Omega(m^{2 ( 1 + 1 / r)}).\qedhere
    \end{equation*}
\end{proof}

Now let $n = 2 m^2$. Taking the disjoint union of $m$ cliques of size $m$ with a path on $n / 2 = m^2$ vertices gives an $n$-vertex graph such that with high probability the algorithm in the proof of \Cref{main-theorem} yields no better than a $(\floor{(1 - o(1))r},\Omega(n^{1/(r+1)}))$-cover for any $2 \le r \le n^{o(1)}$, while there trivially exists a $(O(1),O(1))$-cover.

\section{Lower bound from vertex expansion}\label{other-direction-thm-proof-section}

We show that graphs with strong expansion do not admit good $(r, k)$-covers.

\begin{definition}
    For an $n$-vertex graph $G$ and real numbers $h \geq 1$ and $\alpha \in (0, 1)$, we say $G$ is an \textit{$(\alpha, h)$-expander} if
    \begin{equation*}
        |N^{\leq 1}(W)| > h |W|
    \end{equation*}
    for every non-empty $W \subseteq V(G)$ with $|W| \leq \alpha n$.
\end{definition}

\begin{lemma}\label{lossless-expanders-dont-admit-good-covers}
    Suppose $n \geq 1$ and $G$ is an $n$-vertex $(\alpha, h)$-expander with maximum degree at most $D$.
    Then for any positive integer $r$ with $(D+1)^r \leq \alpha n$, $G$ admits no $(r, h)$-cover.
\end{lemma}
\begin{proof}
    Suppose for the sake of contradiction that $U_1, \dots, U_m$ is an $(r, h)$-cover for $G$.
    For $1 \leq i \leq m$, we have that $G[U_i]$ has radius at most $r$.
    Therefore, $|U_i| = 0$ or there exists a $v \in U_i$ with $U_i = N^{\leq r}_{G[U_i]}(v) \subseteq N^{\leq r}(v)$, which implies
    \begin{equation*}
        |U_i| \leq |N^{\leq r}(v)| \leq (D + 1)^r \leq \alpha n.
    \end{equation*}
    Since $G$ is an $n$-vertex $(\alpha, h)$-expander, it follows that $|N^{\leq 1}(U_i)| > h |U_i|$ if $|U_i| > 0$.
    Consequently,
    \begin{equation*}
        h n \geq \sum_{i = 1}^m |N^{\leq 1}(U_i)| > h \sum_{i = 1}^m |U_i| \geq h n
    \end{equation*}
    since $U_1, \dots, U_m$ cover $G$, a contradiction.
\end{proof}

The previous lemma shows that graphs with expansion $h$ large relative to the maximum degree $D$ do not admit local efficient covers. If $h$ is very close to $D$, such graphs are often called lossless expanders; see the survey of Hoory, Linial, and Wigderson~\cite{expander-survey} for more background on expander graphs. Hsieh, Lubotzky, Mohanty, Reiner, and Zhang~\cite{explicit-lossless-vertex-expanders} recently gave the first construction of explicit lossless expanders. For our purposes, however, the following standard probabilistic construction suffices; we include the calculation after proving \Cref{other-direction-theorem}. 
\begin{lemma}\label{existence-of-lossless-expander}
    Let $D$ be a natural number, let $1 \leq h < D - 1$ be a real number, and set
    \begin{equation*}
        \alpha \coloneqq
            \frac{1}{2h} \left(\frac{1}{4 h e^{h + 1}}\right)^{1 / (D - h - 1)}.
    \end{equation*}
    Then for any natural number $n$, it holds with positive probability that the union of $D$ independent uniformly random perfect matchings in $K_{n, n}$ is a $2n$-vertex $(\alpha, h)$-expander with maximum degree at most $D$.
\end{lemma}

Before proving the lemma, let us use it to prove our efficiency lower bound.
\uncoverable*
\begin{proof}
    Set $h \coloneqq n^{1 / (r+1)}/15$.
    If $h < 1$, then no $n$-vertex graph admits an $(r, h)$-cover, so the statement is trivial.
    Now suppose that $h \geq 1$.
    Then set $D = 2\lceil h + 1\rceil$ and, as in \Cref{existence-of-lossless-expander},
    \begin{equation*}
        \alpha \coloneqq \frac{1}{2h} \left(\frac{1}{4 h e^{h + 1}}\right)^{1 / (D - h - 1)} \geq \frac{1}{2h} \frac{1}{e (4 h)^{1 / (h+1)}} \geq \frac{1}{2 e^2 h} \geq \frac{1}{15 h}
    \end{equation*}
    since $4 h \leq e^2 h \leq e^{h+1}$.
    By applying \Cref{existence-of-lossless-expander} with $n / 2$ in place of $n$, we get that there exists an $n$-vertex $(\alpha, h)$-expander $G$ with maximum degree at most $D$.
    Furthermore,
    \begin{equation*}
        \frac{(D + 1)^r}{\alpha} \leq 15h (2 (h + 2) + 1)^r \leq (15 h)^{r+1} = n.
    \end{equation*}
    Therefore, by \Cref{lossless-expanders-dont-admit-good-covers}, $G$ admits no $(r, h)$-cover.
\end{proof}

\begin{proof}[Proof of \Cref{existence-of-lossless-expander}]
    Denote $\beta \coloneqq 2\alpha$.
    Then $e^{h+1}h(\beta h)^{D-h-1} \leq 1/4$. Let $M_1, \dots, M_D$ be independent uniformly random perfect matchings in the complete bipartite graph $K_{n,n}$ with bipartition $A \dcup B$.
    We will show that $G = (A \dcup B, M_1 \cup \cdots \cup M_D)$ has the desired properties with positive probability.
    Indeed, by construction, $G$ always has $2n$ vertices and maximum degree at most $D$.

    For $S \subseteq A$, write $\Gamma(S) \coloneqq N^{\leq 1}(S) \cap B$, and for $S \subseteq B$, write $\Gamma(S) \coloneqq N^{\leq 1}(S) \cap A$.
    Now fix $W \subseteq A$ of size $w$ with $1 \leq w \leq \beta n$.
    If $|\Gamma(W)| \leq h|W| = hw$, then there is a set $U \subseteq B$ of size $\lfloor hw \rfloor$ such that $\Gamma(W) \subseteq U$.
    For any fixed such $U$, the probability that all $D$ matchings pair every vertex in $W$ with a vertex in $U$ is
    \begin{equation*}
        \left(\binom{\lfloor hw \rfloor}{w} / \binom{n}{w}\right)^D.
    \end{equation*}
    Therefore, by a union bound,
    \begin{equation*}
        \IP(\exists W \subseteq A,\ 1 \leq |W| \leq \beta n,\ |\Gamma(W)| \leq h|W|)
        \leq
        \sum_{w=1}^{\lfloor \beta n \rfloor}
        \binom{n}{w}
        \binom{n}{\lfloor hw \rfloor}
        \left(\binom{\lfloor hw \rfloor}{w} / \binom{n}{w}\right)^D.
    \end{equation*}
    Using the inequality $\binom{a}{b} \leq (e a / b)^b$, valid for all positive integers $a, b$ and the fact that $(e x / y)^y$ is monotonically increasing in $y$ for real numbers $x \geq y > 0$, we get that
    \begin{equation*}
        \sum_{w=1}^{\lfloor \beta n \rfloor}
        \binom{n}{w}
        \binom{n}{\lfloor hw \rfloor}
        \left(\binom{\lfloor hw \rfloor}{w} / \binom{n}{w}\right)^D \leq \sum_{w=1}^{\lfloor \beta n \rfloor}
        \left(\frac{en}{w}\right)^w
        \left(\frac{en}{hw}\right)^{hw}
        \left(\frac{hw}{n}\right)^{Dw}
    \end{equation*}
    The right-hand side equals
    \begin{equation*}
        \sum_{w=1}^{\lfloor \beta n \rfloor}
        \left(e^{h+1}h
        \left(\frac{hw}{n}\right)^{D-h-1}\right)^w \\
        \leq
        \sum_{w=1}^{\lfloor \beta n \rfloor}
        \left(e^{h+1}h(\beta h)^{D-h-1}\right)^w \\
        \leq
        \sum_{w=1}^{\infty}
        \frac{1}{4^w}
        =
        \frac{1}{3}.
    \end{equation*}
    The same estimate holds with $A$ and $B$ interchanged.
    Thus, with probability at least $1 - 2/3 > 0$, the following property holds: for every non-empty set $S \subseteq A$ or $S \subseteq B$ with $|S| \leq \beta n$, we have $|\Gamma(S)| > h|S|$.
    
    Fix an outcome for which this property holds.
    We claim that $G$ is an $(\alpha,h)$-expander.
    Let $W \subseteq A \dcup B$ be non-empty with $|W| \leq \alpha |V(G)| = 2\alpha n = \beta n$. Set $X \coloneqq W \cap A$ and $Y \coloneqq W \cap B$.
    Since $G$ is bipartite,
    \begin{equation*}
        |N^{\leq 1}(W)|
        =
        |X \cup \Gamma(Y)| + |Y \cup \Gamma(X)|
        \geq
        |\Gamma(Y)| + |\Gamma(X)|.
    \end{equation*}
    If $X$ is non-empty, then $|\Gamma(X)| > h|X|$, and if $Y$ is non-empty, then $|\Gamma(Y)| > h|Y|$.
    Since at least one of $X,Y$ is non-empty, it follows that $|N^{\leq 1}(W)| > h|X| + h|Y| = h|W|$. Thus $G$ is an $(\alpha,h)$-expander.
\end{proof}

\bibliographystyle{amsplain}
\bibliography{graph-covers_refs}

\end{document}